\documentclass[12pt]{amsart}

\theoremstyle{plain}
\newtheorem{theorem}{Theorem}[section]
\newtheorem{proposition}[theorem]{Proposition}
\newtheorem{corollary}[theorem]{Corollary}
\newtheorem{lemma}[theorem]{Lemma}

\theoremstyle{definition}
\newtheorem{definition}[theorem]{Definition}
\newtheorem{remark}[theorem]{Remark}
\newtheorem*{tablenote*}{Tables}
\numberwithin{equation}{section}
\newtheorem{example}[theorem]{Example}

\newcommand{\N}{\mathbb{N}}

\usepackage{hyperref}
\usepackage{url}
\usepackage{amsfonts}
\usepackage{amscd}
\usepackage[latin2]{inputenc}
\usepackage{t1enc}
\usepackage[mathscr]{eucal}
\usepackage{indentfirst}
\usepackage{graphicx}
\usepackage{graphics}
\usepackage{pict2e}
\usepackage{epic}
\numberwithin{equation}{section}
\usepackage[margin=2.86cm]{geometry}
\usepackage{epstopdf} 
\usepackage{amsmath}
\usepackage{amsthm}
\usepackage{amssymb}
\usepackage{mathrsfs}
\usepackage{graphicx}
\usepackage{enumerate}
\usepackage{tikz}
\usetikzlibrary{shapes.misc}
\usepackage{float}
\usepackage{comment}

\usepackage{pifont}
\newcommand{\xmark}{\ding{55}}

\tikzset{cross/.style={cross out, draw=black, minimum size=2*(#1-\pgflinewidth), inner sep=0pt, outer sep=0pt}, 
	cross/.default={2pt}}

\makeatletter
\@namedef{subjclassname@2020}{\textup{2020} Mathematics Subject Classification}
\makeatother

\title[Normalizers of Sylows in finite reflection groups]{Normalizers of Sylow subgroups in finite reflection groups}

\author[Kane D. Townsend]{Kane Douglas Townsend}

\address{Mathematical Sciences Institute, Australian National University, Canberra ACT 2601, Australia}

\email{kane.townsend@anu.edu.au}



\begin{document}
	
	\begin{abstract}
		Let $W$ be a finite reflection group, either real or complex, and $S_\ell$ a Sylow $\ell$-subgroup of $W$. We prove the existence of a semidirect product decomposition of $N_W(S_\ell)$ in terms of the unique parabolic subgroup of $W$ minimally containing $S_\ell$ and known decompositions of normalizers of parabolic subgroups. In the real setting, the description follows from the existence of Sylow $\ell$-subgroups stable under the Coxeter diagram automorphisms of finite reflection groups with no proper parabolic subgroup containing a Sylow $\ell$-subgroup. 
	\end{abstract}
	
	\maketitle

	\section{Introduction}\label{intro}
	
	Let $V$ be a finite dimensional vector space, either real or complex. Let $W$ be a finite reflection group on $V$. Let $\ell$ be a prime number dividing the order of $W$ and $\text{Syl}_\ell(W)$ be the unique conjugacy class of Sylow $\ell$-subgroups of $W$. The parabolic and reflection subgroups that \emph{minimally contain} (see Definition \ref{def:minimallycontains}) a Sylow $\ell$-subgroup of $W$ were classified, up to conjugacy, in the real setting in \cite{ME} and the complex setting in \cite{METOO}. The set of parabolic subgroups of $W$, denoted by $\mathcal{P}(W)$, is closed under conjugation and intersection. Hence, for each Sylow $\ell$-subgroup there is a unique parabolic subgroup that minimally contains it (special case of Lemma \ref{lem:nosharing}). Therefore, the parabolic subgroups that minimally contain a Sylow $\ell$-subgroup of $W$ form a unique conjugacy class (see Corollary \ref{cor:parabolicuniqueclass}). In comparison, the set of reflection subgroups of $W$, denoted by $\mathcal{R}(W)$, is closed under conjugation but not intersection. As a result, the reflection subgroups that minimally contain a Sylow $\ell$-subgroup do not have such nice properties.
	
	\begin{definition}\label{def:minimallycontains}
		Let $G$ be a finite group and $\mathcal{S}(G)$ be a subset of the subgroups of $G$. If $H$ is a subgroup of $G$, then we say that $K$ \emph{minimally contains $H$ in $\mathcal{S}(G)$} if $K$ is a minimal element of $\mathcal{S}(G)$, with respect to containment, that contains $H$ as a subgroup. Furthermore, given a prime number $\ell$ we let, \[\mathcal{S}_\ell(G):=\{S\in\mathcal{S}(G) \mid S \ \text{minimally contains an element of} \ \text{Syl}_\ell(G)\}.\] 
	\end{definition}
	
	\begin{definition}\label{def:sylowclassofparabolics}
		Call $\mathcal{P}_\ell(W)$ the \emph{$\ell$-Sylow class of parabolic subgroups} of $W$. If $\mathcal{P}_\ell(W)=\{W\}$, we say that $W$ is \emph{$\ell$-{cuspidal}}.
	\end{definition}
	
	\begin{definition}\label{def:sylowclassofreflection}
		Call $\mathcal{R}_\ell(W)$ the \emph{$\ell$-Sylow classes of reflection subgroups}. If $\mathcal{R}_\ell(W)=\{W\}$, we say that $W$ is \emph{$\ell$-{supercuspidal}}.
	\end{definition}
	
	We note that the wording of the above definitions assumes that $\mathcal{P}_\ell(W)$ forms a unique conjugacy class in $W$ (see Corollary \ref{cor:parabolicuniqueclass}), while $\mathcal{R}_\ell(W)$ may not.
	
	The $\ell$-Sylow class of parabolic subgroups appears in the study of modular representation theory for finite groups of Lie type. In \cite[Theorem 4.5]{AHJR}, the $\ell$-Sylow class of parabolic subgroups of a Weyl group assists in determining the modular generalized Springer correspondence for a connected reductive group. In \cite[Theorem 4.2]{GHM}, the analogous idea of Levi subgroups minimally containing a Sylow $\ell$-subgroup of a finite reductive group are involved in determining the semisimple vertex of the $\ell$-modular Steinberg character.
	
	In \cite[Lemma 2]{HOWLETT}, it was shown that given a reflection subgroup $R$ of a finite real reflection group $W$, that $N_W(R)=R\rtimes U$ for some $U\leq W$, which we call the \emph{H-complement}. We note that the generalization from a parabolic subgroup to a reflection subgroup is mentioned in \cite[Lemma 3.3]{GHM3}, and that $U$ induces Coxeter diagram automorphisms on $R$. In \cite{MURA}, it is shown that given a parabolic subgroup $P$ of a finite complex reflection group $W$, that $N_W(P)=P\rtimes U$ for some $U\leq W$, which we call the \emph{MT-complement}. We will see that given a Sylow $\ell$-subgroup $S_\ell$ of $W$, there exists a unique parabolic subgroup minimally containing it, and so $N_W(S_\ell)\leq N_W(P_\ell)$ (special cases of Lemma \ref{lem:nosharing} and Corollary \ref{cor:normalizercontainment}). It is then natural to ask if the normalizer of a Sylow subgroup in a finite reflection group will have a similar structure to the normalizer of a parabolic subgroup. This leads to the following main results.
	
	\begin{theorem}\label{thm:normalizersylowreal}
		Let $W$ be a finite real reflection group with simple system $\Delta$. Let $W(\Lambda)\in \mathcal{P}_\ell(W)$ with $\Lambda\subseteq \Delta$. Then there exists a Sylow $\ell$-subgroup $S_\ell\leq W(\Lambda)$ such that $N_W(S_\ell)=N_{W(\Lambda)}(S_\ell)\rtimes U_\Lambda$, where $U_\Lambda$ is the H-complement of $W(\Lambda)$.
	\end{theorem}
	
	\begin{theorem}\label{thm:normalizersylowparacomp}
		Let $W$ be a finite complex reflection group and $P_\ell\in \mathcal{P}_\ell(W)$. Then there exists a Sylow $\ell$-subgroup $S_\ell\leq P_\ell$ such that $N_W(S_\ell)=N_{P_\ell}(S_\ell)\rtimes U$, where $U$ is the MT-complement of $P_\ell$.
	\end{theorem}
	
	We recall that by the Schur-Zassenhaus theorem, for a general finite group $G$, the normaliser $N_G(S)$ is the semidirect product $S\rtimes U$, for some complement $U$, and that each complement of $S$ in $N_G(S)$ is conjugate to $U$ in $N_G(S)$. Hence, the above theorems are specializations of the Schur-Zassenhaus theorem to finite reflection groups. 
	
	To prove Theorem \ref{thm:normalizersylowreal}, we prove a preliminary result regarding the existence of Sylow $\ell$-subgroups stable under Coxeter diagram automorphisms.
	
	\begin{theorem}\label{thm:sylowstablediagramautroeal}
		Let $W$ be an $\ell$-cuspidal finite real reflection group. There exists a Sylow $\ell$-subgroup of $W$ stable under the Coxeter diagram automorphisms of $W$ except when $W$ has a component of type $F_4$ and $\ell=3$. In particular, if $W$ is $\ell$-supercuspidal, then there exists a Sylow $\ell$-subgroup stable under the Coxeter diagram automorphisms of $W$.
	\end{theorem}
	
	The structure of the paper is as follows. In Section \ref{Preliminaries}, we introduce definitions and prove some preliminary results. In Section \ref{sylowstableunderauto} we prove Theorem \ref{thm:sylowstablediagramautroeal} using case-by-case arguments. In Section \ref{normalizersofsylowsubgroupsinreal} we prove Theorem \ref{thm:normalizersylowreal}. We also determine when a more refined decomposition exists in terms of reflection subgroups minimally containing Sylow $\ell$-subgroups (see Theorem \ref{thm:normalizersylowref}). In Section \ref{complexcase} we prove Theorem \ref{thm:normalizersylowparacomp} using case-by-case arguments and MAGMA \cite{MAGMA} computations. In Section \ref{complexdiagramsproof} we generalize the notion of Coxeter diagram automorphism to the complex setting and prove an analogue to Theorem \ref{thm:sylowstablediagramautroeal} (see Theorem \ref{thm:sylowstablediagramautocomp}). 
	
	\subsection*{MAGMA Calculations} Code is provided at \url{https://doi.org/10.17605/OSF.IO/8UAM7} for the MAGMA calculations mentioned throughout this paper.
	
	\section{Preliminaries}\label{Preliminaries}
	
	A \emph{reflection} on $V$ is a finite order linear operator that fixes a hyperplane of $V$ pointwise. A \emph{finite reflection group} is a finite group generated by reflections on $V$. A \emph{reflection subgroup} of $W$ is a subgroup generated by reflections. A \emph{parabolic subgroup} of $W$ is the pointwise stabilizer $W_U:=\{w\in W \mid wu=u \ \text{for all} \ u\in U\},$ for some subspace $U$ of $V$. Since $W_{U_1}\cap W_{U_2}=W_{U_1+U_2}$ for subspaces $U_1,U_2\subseteq V$, the intersection of parabolic subgroups of $W$ is a parabolic subgroup of $W$. By \cite[Theorem 1.5]{STEINBERG}, a parabolic subgroup is a reflection subgroup. We now correct an error in the proof of \cite[Lemma 2.3]{ME}, which is used to show that any element of $\mathcal{R}_\ell(W)$ has parabolic closure being an element of $\mathcal{P}_\ell(W)$ (see \cite[Corollary 2.5]{ME}).
	
	\begin{lemma}\label{lem:intersectionpararef}
		The intersection of a parabolic subgroup and reflection subgroup of $W$ is a reflection subgroup of $W$.
	\end{lemma}
	
	\begin{proof}
		Let $P$ be a parabolic subgroup of $W$ and $R$ be a reflection subgroup of $W$. Hence, $P=W_U$ for some subspace $U$ of $V$. Let $X:=\{v\in V \mid wv=v \ \text{for all} \ w\in R\}$. We claim that $P\cap R=R_{U+X}$. Let $w\in P\cap R$. Since $w$ fixes $U$ and $X$ pointwise, it is clear that $w$ fixes $U+X$ pointwise, so $w\in R_{U+X}$. Now let $w\in R_{U+X}$. Then we have $w\in R$ and $wu=u$ for all $u\in U+X$. Since $U\subseteq U+X$, we have $w\in W_U=P$, and conclude $w\in P\cap R$. Therefore, $P\cap R=R_{U+X}$ is a parabolic subgroup of $R$. By \cite[Theorem 1.5]{STEINBERG}, $P\cap R$ is a reflection subgroup of $R$, and hence a reflection subgroup of $W$. 
	\end{proof}
	
	In general, it is not true that the intersection of reflection subgroups of $W$ is a reflection subgroup. For example, in the dihedral group of order $12$ with its standard reflection representation, the non-conjugate dihedral groups of order $6$ are reflection subgroups, but their intersection is a cyclic group of order $3$, which is not a reflection subgroup. 
	
	We now prove some preliminary structural properties of the parabolic subgroups minimally containing a Sylow $\ell$-subgroup.

	\begin{lemma}\label{lem:nosharing}
		Let $H$ be a subgroup of $W$. Then there exists a unique parabolic subgroup of $W$ that minimally contains $H$.
	\end{lemma}
	
	\begin{proof}
		If $H$ is contained in distinct parabolic subgroups $P$ and $P'$, then $P\cap P'\lneq P$ is a parabolic subgroup containing $H$. This contradicts the minimality of $P$ and $P'$.
	\end{proof}
	
	\begin{corollary}\label{cor:parabolicuniqueclass}
		The elements of $\mathcal{P}_\ell(W)$ are conjugate in $W$.
	\end{corollary}
	
	\begin{proof}
		Follows by Lemma \ref{lem:nosharing} and the conjugacy of Sylow $\ell$-subgroups.
	\end{proof}
	
	\begin{corollary}\label{cor:normalizercontainment}
		Let $P_\ell\in\mathcal{P}_\ell(W)$ contain $S_\ell\in \text{Syl}_\ell(W)$. If $H$ is a subgroup of $W$ such that $S_\ell\leq H\leq P_\ell$, then $N_W(H)\leq N_W(P_\ell)$.
	\end{corollary}
	
	\begin{proof}
		If $w\in N_W(H)$, then $H\leq wP_\ell w^{-1}$ and the result follows from Lemma \ref{lem:nosharing}, since $P_\ell$ is the unique element of $\mathcal{P}_\ell(W)$ containing $H$.
	\end{proof}
	
	Special cases of Corollary \ref{cor:normalizercontainment} are when $H=S_\ell\in \text{Syl}_\ell(W)$, and so $N_W(S_\ell)\leq N_W(P_\ell)$, or when $H=R_\ell\in\mathcal{R}_\ell(W)$, and so $N_W(R_\ell)\leq N_W(P_\ell)$.
	
	We will now introduce some standard notation in the specific setting of finite real reflection groups, referring the reader to \cite[Chap. 1-2]{JEH} for details. For the rest of this section as well as Sections \ref{sylowstableunderauto} and \ref{normalizersofsylowsubgroupsinreal}, let $W$ be a finite real reflection group acting on the real vector space $V=\mathbb{R}^n$ for some $n\in \N$. Let $\Phi\subseteq V$ be a root system of the finite real reflection group $W$. The reflection associated to the root $\alpha\in \Phi$ is denoted by $s_\alpha$. Let $\Delta\subseteq \Phi$ be a simple system with set of simple reflections $\{s_\alpha \mid \alpha\in \Delta\}$ generating $W$. Given a simple system $\Delta\subseteq\Phi$, a \emph{standard parabolic subgroup} is a reflection subgroup generated by the reflections associated to the roots in some subset $\Lambda\subseteq \Delta$. The parabolic subgroups of $W$ are $W$-conjugates of standard parabolic subgroups. The standard parabolic subgroups of $W$ correspond to the subgraphs of its Coxeter diagram. The irreducible finite reflection groups have a well-known classification in terms of connected positive-definite Coxeter diagrams. This classification can be found in standard texts on reflection groups and Coxeter groups such as \cite[2.5-2.7]{JEH}. The Coxeter diagrams are included in Figure~\ref{fig:coxdiagrams}, with labels on the vertices to order the simple roots. We refer to types $A_{n-1}, B_n$ and $D_n$ as \emph{classical cases} and the others as \emph{exceptional cases}.
	
	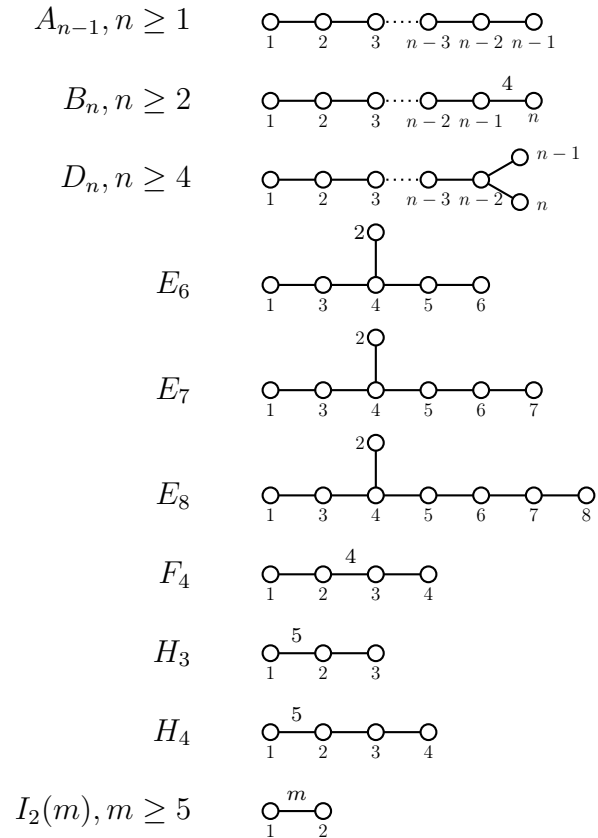
\begin{figure}[h!]
		\caption{Coxeter diagrams of irreducible finite reflection groups}
		\begin{center}
			\begin{tikzpicture}[scale=.35]
				
				\draw (-2.6,0) node[anchor=east]  {$A_{n-1}, n\geq1$};
				\foreach \x in {0,...,5}
				\draw[xshift=\x cm,thick] (\x cm,0) circle (.3cm);
				\draw[dotted,thick] (4.4 cm,0) -- +(1.4 cm,0);
				\foreach \y in {0.15,1.15,3.15,4.15}
				\draw[xshift=\y cm,thick] (\y cm,0) -- +(1.4 cm,0);
				\draw (0,-0.1) node[below] {\scalebox{0.6}{1}};
				\draw (2,-0.1) node[below] {\scalebox{0.6}{2}};
				\draw (4,-0.1) node[below] {\scalebox{0.6}{3}};
				\draw (6,-0.1) node[below] {\scalebox{0.6}{$n-3$}};
				\draw (8,-0.1) node[below] {\scalebox{0.6}{$n-2$}};
				\draw (10,-0.1) node[below] {\scalebox{0.6}{$n-1$}};
				
				\draw (-2.6,-3) node[anchor=east]  {$B_n, n\geq 2$};
				\foreach \x in {0,...,5}
				\draw[xshift=\x cm,thick] (\x cm,-3) circle (.3cm);
				\draw[dotted,thick] (4.4 cm,-3) -- +(1.4 cm,0);
				\foreach \y in {0.15,1.15,3.15,4.15}
				\draw[xshift=\y cm,thick] (\y cm,-3) -- +(1.4 cm,0);
				\draw (9,-3) node[above] {\tiny{4}};
				\draw (0,-3.1) node[below] {\scalebox{0.6}{1}};
				\draw (2,-3.1) node[below] {\scalebox{0.6}{2}};
				\draw (4,-3.1) node[below] {\scalebox{0.6}{3}};
				\draw (6,-3.1) node[below] {\scalebox{0.6}{$n-2$}};
				\draw (8,-3.1) node[below] {\scalebox{0.6}{$n-1$}};
				\draw (10,-3.1) node[below] {\scalebox{0.6}{$n$}};
				
				\draw (-2.6,-6) node[anchor=east]  {$D_n, n\geq4$};
				\foreach \x in {0,1,...,4}
				\draw[xshift=\x cm,thick] (\x cm,-6) circle (.3cm);
				\draw[xshift=8 cm,yshift=-6 cm,thick] (30: 17 mm) circle (.3cm);
				\draw[xshift=8 cm,yshift=-6 cm,thick] (-30: 17 mm) circle (.3cm);
				\draw[dotted,thick] (4.4 cm,-6) -- +(1.4 cm,0);
				\foreach \y in {0.15,1.15,3.15}
				\draw[xshift=\y cm,thick] (\y cm,-6) -- +(1.4 cm,0);
				\draw[xshift=8 cm,yshift=-6 cm,thick] (30: 3 mm) -- (30: 14 mm);
				\draw[xshift=8 cm,yshift=-6 cm,thick] (-30: 3 mm) -- (-30: 14 mm);	
				\draw (0,-6.1) node[below] {\scalebox{0.6}{1}};
				\draw (2,-6.1) node[below] {\scalebox{0.6}{2}};
				\draw (4,-6.1) node[below] {\scalebox{0.6}{3}};
				\draw (6,-6.1) node[below] {\scalebox{0.6}{$n-3$}};
				\draw (8,-6.1) node[below] {\scalebox{0.6}{$n-2$}};
				\draw (9.7,-5) node[right] {\scalebox{0.6}{$n-1$}};
				\draw (9.7,-7) node[right] {\scalebox{0.6}{$n$}};
				
				\draw (-2.6,-10) node[anchor=east]  {$E_6$};
				\foreach \x in {0,...,4}
				\draw[thick,xshift=\x cm] (\x cm,-10) circle (3 mm);
				\foreach \y in {0,...,3}
				\draw[thick,xshift=\y cm] (\y cm,-10) ++(.3 cm, 0) -- +(14 mm,0);
				\draw[thick] (4 cm,-8 cm) circle (3 mm);
				\draw[thick] (4 cm, -9.7 cm) -- +(0, 1.4 cm);
				\draw (2,-10.1) node[below] {\scalebox{0.6}{3}};
				\draw (4,-10.1) node[below] {\scalebox{0.6}{4}};
				\draw (6,-10.1) node[below] {\scalebox{0.6}{5}};
				\draw (8,-10.1) node[below] {\scalebox{0.6}{6}};
				\draw (0,-10.1) node[below] {\scalebox{0.6}{1}};
				\draw (4,-8) node[left] {\tiny{2}};
				
				\draw (-2.6,-14) node[anchor=east]  {$E_7$};
				\foreach \x in {0,...,5}
				\draw[thick,xshift=\x cm] (\x cm,-14) circle (3 mm);
				\foreach \y in {0,...,4}
				\draw[thick,xshift=\y cm] (\y cm,-14) ++(.3 cm, 0) -- +(14 mm,0);
				\draw[thick] (4 cm,-12 cm) circle (3 mm);
				\draw[thick] (4 cm, -13.7 cm) -- +(0, 1.4 cm);
				\draw (2,-14.1) node[below] {\scalebox{0.6}{3}};
				\draw (4,-14.1) node[below] {\scalebox{0.6}{4}};
				\draw (6,-14.1) node[below] {\scalebox{0.6}{5}};
				\draw (8,-14.1) node[below] {\scalebox{0.6}{6}};
				\draw (0,-14.1) node[below] {\scalebox{0.6}{1}};
				\draw (4,-12) node[left] {\scalebox{0.6}{2}};
				\draw (10,-14.1) node[below] {\scalebox{0.6}{7}};
				
				\draw (-2.6,-18) node[anchor=east]  {$E_8$};
				\foreach \x in {0,1,...,6}
				\draw[thick,xshift=\x cm] (\x cm,-18) circle (3 mm);
				\foreach \y in {0,1,...,5}
				\draw[thick,xshift=\y cm] (\y cm,-18) ++(.3 cm, 0) -- +(14 mm,0);
				\draw[thick] (4 cm,-16 cm) circle (3 mm);
				\draw[thick] (4 cm, -17.7 cm) -- +(0, 1.4 cm);
				\draw (2,-18.1) node[below] {\scalebox{0.6}{3}};
				\draw (4,-18.1) node[below] {\scalebox{0.6}{4}};
				\draw (6,-18.1) node[below] {\scalebox{0.6}{5}};
				\draw (8,-18.1) node[below] {\scalebox{0.6}{6}};
				\draw (0,-18.1) node[below] {\scalebox{0.6}{1}};
				\draw (4,-16) node[left] {\scalebox{0.6}{2}};
				\draw (10,-18.1) node[below] {\scalebox{0.6}{7}};
				\draw (12,-18.1) node[below] {\scalebox{0.6}{8}};
				
				\draw (-2.6,-21) node[anchor=east]  {$F_4$};
				\foreach \x in {0,...,3}
				\draw[xshift=\x cm,thick] (\x cm,-21) circle (.3cm);
				\foreach \y in {0.15,2.15}
				\draw[xshift=\y cm,thick] (\y cm,-21) -- +(1.4 cm,0);
				\draw[thick] (2.3 cm,-21) --  +(1.41 cm,0);
				\draw (3.05,-21) node[above] {\tiny{4}};
				\draw (0,-21.1) node[below] {\scalebox{0.6}{1}};
				\draw (2,-21.1) node[below] {\scalebox{0.6}{2}};
				\draw (4,-21.1) node[below] {\scalebox{0.6}{3}};
				\draw (6,-21.1) node[below] {\scalebox{0.6}{4}};
				
				\draw (-2.6,-24) node[anchor=east]  {$H_3$};
				\foreach \x in {0,1,2}
				\draw[xshift=\x cm,thick] (\x cm,-24) circle (.3cm);
				\draw[thick] (0.3,-24) -- +(1.41 cm,0);
				\draw[thick] (2.3 cm,-24) -- +(1.41 cm,0);
				\draw (1,-24) node[above] {\tiny{5}};
				\draw (0,-24.1) node[below] {\scalebox{0.6}{1}};
				\draw (2,-24.1) node[below] {\scalebox{0.6}{2}};
				\draw (4,-24.1) node[below] {\scalebox{0.6}{3}};
				
				\draw (-2.6,-27) node[anchor=east]  {$H_4$};
				\foreach \x in {0,1,2,3}
				\draw[xshift=\x cm,thick] (\x cm,-27) circle (.3cm);
				\draw[thick] (0.3,-27) -- +(1.41 cm,0);
				\draw[thick] (2.3 cm,-27) -- +(1.41 cm,0);
				\draw[thick] (4.3 cm,-27) -- +(1.41 cm,0);
				\draw (1,-27) node[above] {\tiny{5}};
				\draw (0,-27.1) node[below] {\scalebox{0.6}{1}};
				\draw (2,-27.1) node[below] {\scalebox{0.6}{2}};
				\draw (4,-27.1) node[below] {\scalebox{0.6}{3}};
				\draw (6,-27.1) node[below] {\scalebox{0.6}{4}};
				
				\draw (-2.6,-30) node[anchor=east]  {$I_2(m), m\geq5$};
				\foreach \x in {0,1}
				\draw[xshift=\x cm,thick] (\x cm,-30) circle (.3cm);
				\draw[thick] (0.3 cm,-30) -- +(1.41 cm,0);
				\draw (1,-30) node[above] {\tiny{$m$}};
				\draw (0,-30.1) node[below] {\scalebox{0.6}{1}};
				\draw (2,-30.1) node[below] {\scalebox{0.6}{2}};
			\end{tikzpicture}
		\end{center}
		\label{fig:coxdiagrams}
	\end{figure}
	
	The non-trivial Coxeter diagram automorphisms $\rho$ of an irreducible $W$ with a fixed simple system are characterized by the type of $W$ and the order of $\rho$. They are given by:
	\begin{itemize}
		\item[(i)] $A_{n-1}$ with $n\geq 3$ and $|\rho|=2$.
		\item[(ii)] $D_n$ with $n\geq 4$ and $|\rho|=2$.
		\item[(iii)] $D_4$ and $|\rho|=3$.
		\item[(iv)] $I_2(m)$ with $m\geq 4$ and $|\rho|=2$.
		\item[(v)] $F_4$ with $|\rho|=2$. 
		\item[(vi)] $E_6$ with $|\rho|=2$.
	\end{itemize}
	These Coxeter diagram automorphisms will play an important role in proving Theorem~\ref{thm:normalizersylowreal} via their connection to the H-complement, which we will now describe. 
	
	A finite reflection group generated by the reflections associated to a set of vectors $X\subseteq V$ will be denoted by $W(X)$. Fix $W$ with root system $\Phi$ and simple system $\Delta$. Then $W=W(\Phi)=W(\Delta)$. The structure of the normalizer of a reflection subgroup $W(\Lambda)\leq W$ with simple system $\Lambda\subseteq \Phi$ has decomposition $N_W(W(\Lambda))=W(\Lambda)\rtimes U_\Lambda,$ where $U_\Lambda:=\{w\in W \mid w\Lambda=\Lambda\}$. This was shown in \cite[Corollary 3]{HOWLETT} for parabolic subgroups. The statement also holds for reflection subgroups, as noted in \cite[Remark 3.5]{GHM3}. It is clear that $U_\Lambda$ induces Coxeter diagram automorphisms on the reflection subgroups. We note that in \cite{HOWLETT}, $U_\Lambda$ is described in detail for every irreducible finite real reflection group. 
	
	Finally, we have the following elementary result regarding the stability of Sylow $\ell$-subgroups under group automorphisms.
	
	\begin{lemma}\label{lem:groupautosylow}
		Let $\rho$ a group automorphism of a finite group $G$. If $|\rho|=\ell^i$ for some $i\in\N$ and prime $\ell$, then there exists an $S_\ell\in\text{Syl}_\ell(G)$ such that $\rho(S_\ell)=S_\ell$.
	\end{lemma}
	\begin{proof}
		The size of orbits of $\rho$ acting on $\text{Syl}_\ell(G)$ are in $\{\ell^j\mid  0 \leq j\leq i\}$.  Since $|\text{Syl}_\ell(G)|=k\ell+1$ for some $k\in\N$, the result follows.
	\end{proof}

	\section{Sylow subgroups stable under Coxeter diagram automorphisms }\label{sylowstableunderauto}
	
	In this section we investigate the existence of Sylow $\ell$-subgroups of finite real reflection groups stable under Coxeter diagram automorphisms. This will lead us to a proof of Theorem \ref{thm:sylowstablediagramautroeal}. Throughout this section let $W$ be a finite real reflection group.
	
	\begin{example}
		Consider $W_1$ of type $E_6$ with $\rho_{E_6}$ being the order $2$ Coxeter diagram automorphism. Then straightforward MAGMA calculations show that there are no Sylow $5$-subgroups of $W_1$ stable under $\rho_{E_6}$. In contrast, consider $W_2$ of type $A_4$. The Coxeter diagram automorphism $\rho_{A_4}$ will normalize the Sylow $5$-subgroup of $W_1$ given by $\langle s_1s_2s_3s_4 \rangle$. A key difference between $W_1$ and $W_2$ is that $W_1$ is not $5$-cuspidal, while $W_2$ is $5$-cuspidal.
	\end{example}
	
	In fact there are many examples of irreducible $W$ that have no Sylow $\ell$-subgroup stable under its Coxeter diagram automorphisms. Motivated by this example, we ask if all $\ell$-cuspidal finite real reflection groups have a Sylow $\ell$-subgroup stable under the Coxeter diagram automorphisms. This turns out to be almost true (see Theorem \ref{thm:sylowstablediagramautroeal}).

	\begin{remark}\label{rem:reductiontoirred}
		To prove Theorem \ref{thm:sylowstablediagramautroeal}, it is sufficient to determine the existence of a Sylow $\ell$-subgroup stable under Coxeter diagram automorphism for each irreducible $\ell$-cuspidal finite real reflection group. This is illustrated by the following reasoning. Let $W$ have decomposition $W_1^{n_1}\times\dots\times W_k^{n_k}$ into irreducible components and suppose each irreducible component has a Sylow $\ell$-subgroup stable under its Coxeter diagram automorphisms. We note that if $W$ is $\ell$-cuspidal, then so is each of its irreducible factors. For each $1\leq i\leq k$, select a Sylow $\ell$-subgroup $^{(1)}S_i\in \text{Syl}_\ell(W_i)$ that is stable under all Coxeter diagram automorphisms of this copy of $W_i$. For each $1\leq i\leq k$, permute $^{(1)}S_i$ to the other $n_i-1$ copies of $W_i$ via Coxeter diagram automorphisms that switch components of the same type. This gives Sylow $\ell$-subgroups $^{(j_i)}S_i$ of each copy of $W_i$, where $1\leq j_i\leq n_i$. A Sylow $\ell$-subgroup in $W$ that is stable under all Coxeter diagram automorphism of $W$ is then given by $ \prod_{i=1}^{k} {\prod_{j_i=1}^{n_i}} ^{(j_i)}S_i$.
	\end{remark}
	
	Let us now work towards proving Theorem \ref{thm:sylowstablediagramautroeal}. The irreducible $\ell$-cuspidal finite reflection groups are classified in \cite[Table 2]{ME}. Of these cases, the types of $W$ with non-trivial Coxeter diagram automorphisms are the following:
	\begin{itemize}
		\item[(a)] $A_{\ell^i-1}$ for all $\ell$.
		\item[(b)] $D_n$ for $\ell=2$.
		\item[(c)] $I_2(m)$ for $\ell=2$ when $m$ is even and $\ell>2$ for all $m$.
		\item[(d)] $F_4$ for $\ell=2,3$.
		\item[(e)] $E_6$ for $\ell=3$.
	\end{itemize}
	Note that we have included type $B_2$ as $I_2(m)$ with $m=4$. Lemma \ref{lem:groupautosylow} resolves the cases (a) when $\ell=2$, (b) when $n>4$, (c) when $\ell=2$, and (d) when $\ell=2$. For case (d) when $W$ is of type $F_4$ and $\ell=3$, it can be checked using MAGMA that no Sylow $3$-subgroup is stable under $\rho$. 
	
	Consider the remaining part of case (a). We will determine a Sylow $\ell$-subgroup of $W$ that is stable under the order two Coxeter diagram automorphism of $W$. The finite reflection group of type $A_{n-1}$ is isomorphic to the symmetric group, which we denote by $\sigma_n$. The canonical isomorphism maps the simple reflections $s_i$ for $1\leq i\leq n-1$ to the transpositions $(i \; i+1)$. The following well-known description of the Sylow $\ell$-subgroups of a symmetric group (stated below and found in \cite[Pg. 82]{HALL}) gives a useful understanding to address this case. Throughout, let the base-$\ell$ expression of $n$ be $(b_kb_{k-1}\dots b_1b_0)_\ell$.
	
	\begin{proposition}\label{prop:sylowofsymmetric}
		The Sylow $\ell$-subgroups of $\sigma_n$ are isomorphic to $\prod_{i=1}^{k} [\mathcal{C}_{\ell}^{(i)}]^{b_i}$,
		where $\mathcal{C}_{\ell}^{(i)}$ is the iterated wreath product of $i$ copies of the cyclic group of order $\ell$. In particular, $\mathcal{C}_{\ell}^{(i)}$ is a Sylow $\ell$-subgroup of $\sigma_{\ell^i}$.
	\end{proposition}
	
	\begin{proposition}\label{prop:sylowinsymmetricisnormalized}
		Let $W$ be of type $A_{\ell^i-1}$ with $\ell>2$ and $i\geq1$. Then there is a Sylow $\ell$-subgroup of $W$ stable under the order two Coxeter diagram automorphism.
	\end{proposition}
	
	\begin{proof}
		We use the canonical isomorphism between $W$ and $\sigma_{\ell^i}$. By Proposition \ref{prop:sylowofsymmetric}, a Sylow $\ell$-subgroup of $\sigma_{\ell^i}$ is isomorphic to $\mathcal{C}_{\ell}^{(i)}$. The order two Coxeter diagram automorphism under the canonical isomorphism acts via conjugation by \[\rho_i:=(1 \; \ell^i)(2 \; \ell^i-1)\dots (\lfloor\frac{\ell^i}{2}\rfloor \; \lceil\frac{\ell^i}{2}\rceil+1).\] We inductively build a Sylow $\ell$-subgroup that is stable under $\rho_i$. When $i=1$, take the cycle $(1 \; 2 \; \dots \; \ell)$ as the generator of the Sylow $\ell$-subgroup stable under $\rho_1$. Now assume $i>1$ and partition $\{1,2,..,\ell^i\}$ into $\ell$ blocks of size $\ell^{i-1}$ given by $X_j:=\{1+j\ell^{i-1},2+j\ell^{i-1},\dots,\ell^{i-1}+j\ell^{i-1}\}$ for $0\leq j\leq \ell-1$. Let $^{(j)}\sigma_{\ell^{i-1}}$ be the symmetric group permuting $X_j$. Suppose there is a Sylow $\ell$-subgroup $S_0$ of $^{(0)}\sigma_{\ell^{i-1}}$ stable under $\rho_{i-1}$. For each $0\leq j\leq \frac{\ell-1}{2}$, conjugate $S_0$ by \[(1\;1+j\ell^{i-1})(2\;2+j\ell^{i-1})\dots(\ell^{i-1}\;\ell^{i-1}+j\ell^{i-1})\] to give a Sylow $\ell$-subgroup $S_j$ of $^{(j)}\sigma_{\ell^{i-1}}$. For each $0\leq j< \frac{\ell-1}{2}$, conjugate $S_j$ by $\rho_i$ to give a Sylow $\ell$-subgroup $S_{\ell-j}$ of $^{(\ell-j)}\sigma_{\ell^{i-1}}$. By our construction, $S_{\frac{\ell-1}{2}}$ is stable under $\rho_i$. The group generated by each $S_j$ for $0\leq j\leq \ell-1$ as well as \[(1\; 1+\ell^{i-1} \; \dots \; 1+(\ell-1)\ell^{i-1})(2\; 2+\ell^{i-1} \; \dots \; 2+(\ell-1)\ell^{i-1})\dots(\ell^{i-1} \; 2\ell^{i-1}\;\dots\;\ell^i)\] is a Sylow $\ell$-subgroup of $\sigma_{\ell^i}$ stable under $\rho_i$.
	\end{proof}
	
	For clarity, we give an example of the Sylow $\ell$-subgroup constructed in the proof of Proposition \ref{prop:sylowinsymmetricisnormalized} when $i=2$ and $\ell=3$.
	
	\begin{example}
		Consider $W$ of type $A_8$ and its canonical isomorphism to $\sigma_{9}$. The order two Coxeter diagram automorphism acts as conjugation by $\rho=(1\;9)(2\;8)(3\;7)(4\;6)$. Select the Sylow $3$-subgroup $S_0$ of $\sigma_3$ acting on $\{1,2,3\}$ generated by $(1\; 2 \; 3)$. Then conjugate $S_0$ by $(1\; 4)(2\;5)(3\;6)$ to get a Sylow $3$-subgroup $S_1$ of $\sigma_3$ acting on $\{4,5,6\}$ generated by $(4\; 5 \; 6)$. Then conjugating $S_0$ by $\rho$ we get a Sylow $3$-subgroup $S_2$ of $\sigma_3$ acting on $\{7,8,9\}$ that is generated by $(9\;8\;7)$. Hence, a Sylow $3$-subgroup of $\sigma_9$ is given by \[\langle(1\;2\;3),(4\;5\;6),(9\;8\;7),(1\;4\;7)(2\;5\;8)(3\;6\;9)\rangle.\] Conjugation by $\rho$ on the generators sends $(1\;2\;3)\mapsto (9\;8\;7)$, $(4\;5\;6)\mapsto(4\;5\;6)^{-1}$, $(9\;8\;7)\mapsto (1\;2\;3)$ and $[(1\;4\;7)(2\;5\;8)(3\;6\;9)]\mapsto[(1\;4\;7)(2\;5\;8)(3\;6\;9)]^{-1}$. Hence, this Sylow $3$-subgroup of $W$ is stable under $\rho$.
	\end{example}

	\begin{proof}[Proof of Theorem \ref{thm:sylowstablediagramautroeal}]
		By previous observations in this section and Proposition \ref{prop:sylowinsymmetricisnormalized}, the remaining irreducible $\ell$-cuspidal cases to consider are $D_4$ when $\ell=2$, $I_2(m)$ when $\ell>2$, and $E_6$ when $\ell=3$. We will now give a Sylow $\ell$-subgroup stable under the diagram automorphisms of each of these cases.
		
		Let $W$ be of type $D_4$. Let $\rho_1$ be the order two Coxeter diagram automorphism given by:  \[\alpha_1\mapsto \alpha_1 \hspace{1em} \alpha_2\mapsto \alpha_2 \hspace{1em} \alpha_3\mapsto \alpha_4 \hspace{1em} \alpha_4\mapsto \alpha_3\] Let $\rho_2$ be the order three Coxeter diagram automorphism given by: \[\alpha_1\mapsto \alpha_3 \hspace{1em} \alpha_2\mapsto \alpha_2 \hspace{1em} \alpha_3\mapsto \alpha_4 \hspace{1em} \alpha_4\mapsto \alpha_1\]
		It is easy to check that $S_2=\langle s_1,s_3,s_2s_1s_3s_2,s_2s_1s_4s_2\rangle$ is a Sylow $2$-subgroup of $W$ stable under both $\rho_1$ and $\rho_2$.
		
		Now let $W$ be of type $I_2(m)$ and $\ell>2$. Then the unique Sylow $\ell$-subgroup $\langle (s_1s_2)^\frac{m}{\nu_\ell(m)}  \rangle$ is stable under the order two Coxeter diagram automorphism switching $\alpha_1$ and $\alpha_2$.
		
		Now let $W$ be of type $E_6$ with simple system $\Delta=\{\alpha_1,\dots,\alpha_6\}$ corresponding to the numbering in Figure \ref{fig:coxdiagrams}. Then the only non-trivial Coxeter diagram automorphism is of order two. This automorphism permutes the indices of the roots via $(1\; 6)(3\; 5)$. A Sylow $3$-subgroup of $W$ is given by $S_3=\langle s_1s_3,s_0s_2,s_6s_5,w_0'w_0\rangle$,
		where $w_0$ is the longest element of $E_6$ with respect to $\Delta$, $w_0'$ is the longest element of $E_6$ with respect to the simple system $\{\alpha_2,\alpha_3,\alpha_4,\alpha_5,\alpha_6,-\alpha_0\}$ and $\alpha_0$ is the highest root with respect to $\Delta$. The element $w_0'w_0$ has order three and maps the simple roots as follows:
		\[\alpha_1 \mapsto -\alpha_0 \hspace{1em} \alpha_2 \mapsto\alpha_5 \hspace{1em} \alpha_3 \mapsto \alpha_2 \hspace{1em} \alpha_4 \mapsto \alpha_4 \hspace{1em} \alpha_5 \mapsto \alpha_3 \hspace{1em} \alpha_6  \mapsto \alpha_1\]
		The order two Coxeter diagram automorphism acts on the generators of $S_3$ as follows:
		\[s_1s_3 \mapsto s_6s_5 \hspace{1em} s_0s_2\mapsto s_0s_2 \hspace{1em} s_6s_5\mapsto s_1s_3 \hspace{1em} w_0'w_0\mapsto (w_0'w_0)^{-1}.\] Hence, $S_3$ is stable under the order two Coxeter diagram automorphism.
		
		Since we have confirmed that all $\ell$-cuspidal irreducible finite real reflection groups have a Sylow $\ell$-subgroups stable under the diagram automorphisms except for type $F_4$ when $\ell=3$, the first part of Theorem \ref{thm:sylowstablediagramautroeal} follows. If $W$ is $\ell$-supercuspidal, then it is also $\ell$-cuspidal. The second part of Theorem \ref{thm:sylowstablediagramautroeal} follows from the observation that type $F_4$ is not $3$-supercuspidal; it has a reflection subgroup of type $A_2\times \tilde{A}_2$ that contains a Sylow $3$-subgroup.
	\end{proof}

	\section{normalizers of Sylow subgroups in finite real reflection groups}\label{normalizersofsylowsubgroupsinreal}
	
	In this section we prove the existence of a semidirect product decomposition of the normalizer of a Sylow subgroup in any finite real reflection group. This decomposition is found by combining results from \cite{HOWLETT} and Theorem \ref{thm:sylowstablediagramautroeal}.

	\begin{proof}[Proof of Theorem \ref{thm:normalizersylowreal}]
		By Corollary \ref{cor:normalizercontainment} and \cite{HOWLETT}, $N_W(S_\ell)\leq N_W(W(\Lambda))=W(\Lambda)\rtimes U_\Lambda$, where $U_\Lambda=\{w\in W \mid w\Lambda=\Lambda\}$. By similar reasoning to Remark \ref{rem:reductiontoirred} it is sufficient to show that for each irreducible finite real reflection group $W$ there exists a Sylow $\ell$-subgroup $S_\ell$ stable under the Coxeter diagram automorphisms of all irreducible components of $P_\ell\in \mathcal{P}_\ell(W)$, where $S_\ell\leq P_\ell\lneq W$ (if $P_\ell=W$ then the result is trivial). By Theorem \ref{thm:sylowstablediagramautroeal}, the only possible exception is if $P_\ell$ is of type $F_4$ when $\ell=3$ and is a proper parabolic subgroup of $W$. By inspection of Coxeter diagrams, a proper parabolic subgroup of an irreducible finite real reflection group can never be of type $F_4$, so the result follows.
	\end{proof}
	
	By the conjugacy of Sylow $\ell$-subgroups in $W$, any Sylow $\ell$-subgroup will have such a decomposition of its normalizer for an appropriate choice of simple system of $W$. This means that describing normalizers of Sylow $\ell$-subgroups of finite reflection groups can be reduced to the situation that $W$ is $\ell$-cuspidal. The minimality of a parabolic subgroup containing a Sylow $\ell$-subgroup is essential for both Corollary \ref{cor:normalizercontainment} and Theorem \ref{thm:normalizersylowreal} to hold, as shown by the following example.
	
	\begin{example}
		Let $W$ be of type $E_7$. Then the standard parabolic subgroup $W(\Lambda)$ of type $D_5$, where $\Lambda=\{\alpha_1,\alpha_2,\alpha_3,\alpha_4,\alpha_5\}$ with respect to the ordering in Figure \ref{fig:coxdiagrams}, contains Sylow $5$-subgroups of $W$. However, it is not minimal with respect to this property, since parabolic subgroups of type $A_4$ will contain a Sylow $5$-subgroup. If we let $S_5=\langle s_1s_3s_4s_2\rangle\in \text{Syl}_5(W(\Lambda))$, then $s_6\in N_W(S_5)$ but $s_6\notin N_W(W(\Lambda))$. Hence, there is no $S\in \text{Syl}_5(W(\Lambda))$ with $N_W(S)\leq N_W(W(\Lambda))$. MAGMA calculations also confirm that no Sylow $5$-subgroup in $W(\Lambda)$ is normalized by $U_\Lambda=\langle \rho, s_7\rangle$, where $\rho$ induces the order two Coxeter diagram automorphism of $W(\Lambda)$ switching $\alpha_2$ and $\alpha_5$.
	\end{example}
	
	We will now refine the decomposition Theorem \ref{thm:normalizersylowreal} by replacing elements of $\mathcal{P}_\ell(W)$ by elements of $\mathcal{R}_\ell(W)$. There is no analogue of Corollary \ref{cor:normalizercontainment} for elements of $\mathcal{R}_\ell(W)$. Hence, we will first need to identify when
	\begin{equation}\label{suitability}
		N_W(S_\ell)\leq N_W(W(\Lambda))
	\end{equation}
	holds, where $S_\ell\in\text{Syl}_\ell(W)$ is a subgroup of $W(\Lambda)\in \mathcal{R}_\ell(W)$. When (\ref{suitability}) holds we say that $W$ is \emph{$\ell$-suitable}, otherwise we say $W$ is \emph{$\ell$-unsuitable}.
	
	\begin{theorem}\label{thm:normalizersylowref}
		Let $W$ be an $\ell$-suitable finite real reflection group and $W(\Lambda)\in\mathcal{R}_\ell(W)$ with $\Lambda\subseteq\Phi$ a simple system for $W(\Lambda)$. Then there exists a Sylow $\ell$-subgroup $S_\ell\leq W(\Lambda)$ such that $N_W(S_\ell)=N_{W(\Lambda)}(S_\ell)\rtimes U_\Lambda$.
	\end{theorem}
	
	\begin{proof}
		By assumption $W$ is $\ell$-suitable, so by \cite{HOWLETT}, $N_W(S_\ell)\leq N_W(W(\Lambda))=W(\Lambda)\rtimes U_\Lambda$, where $U_\Lambda=\{w\in W \mid w\Lambda=\Lambda\}$. It is sufficient to show that for each irreducible finite real reflection group $W$ there exists a Sylow $\ell$-subgroup $S_\ell$ stable under the Coxeter diagram automorphisms of all irreducible components of $R_\ell\in \mathcal{R}_\ell(W)$, where $S_\ell\leq R_\ell\lneq W$ (if $R_\ell=W$, then the result is trivial). By Theorem \ref{thm:sylowstablediagramautroeal}, such a Sylow $\ell$-subgroup always exists.
	\end{proof}
	
	By the conjugacy of Sylow $\ell$-subgroups of $W$, any Sylow $\ell$-subgroup will have such a decomposition for an appropriate choice of simple system of $W(\Lambda)$. Therefore, when $W$ is $\ell$-suitable, describing normalizers of Sylow $\ell$-subgroups of finite reflection groups can be reduced to the situation that $W$ is $\ell$-supercuspidal.
	
	To determine the usefulness of Theorem \ref{thm:normalizersylowref}, we must identify which finite real reflection groups are $\ell$-suitable. It is sufficient to determine $\ell$-suitability for irreducible $W$ that are not $\ell$-supercuspidal. Furthermore, in \cite[Observation 1.4]{METOO} it is noted that if $W$ is irreducible and not $\ell$-cuspidal, then $\mathcal{P}_\ell(W)=\mathcal{R}_{\ell}(W)$. Hence, any $W$ that is not $\ell$-cuspidal will automatically be $\ell$-suitable by Corollary \ref{cor:normalizercontainment}. Therefore, it is sufficient to check $\ell$-suitability for irreducible $W$ that is $\ell$-cuspidal and not $\ell$-supercuspidal. These $W$ can be deduced from \cite[Table 3]{ME}. We will determine the irreducible finite reflection groups that are $\ell$-suitable and $\ell$-unsuitable cases, summarising the results in Table \ref{table:nonsupercuspidalreal}.

	\begin{table}[H]
		\begin{center}
			\caption{$\ell$-cuspidal irreducible $W$ that are not $\ell$-supercuspidal}
			\scalebox{0.88}{\begin{tabular}{ | c | c | c | }
					\hline
					{Type of $W$} & $\ell$ & Type of $W(\Lambda)\in \mathcal{R}_\ell(W)$    \\ \hline \hline			
					$B_n$ & $2$ & $\prod_{i=0}^{k}B_{2^i}^{b_i}$   \\ \hline
					
					$I_2(m), m$ even & $2$ & $I_2(2^{\nu_2(m)})$  \\ \hline
					
					$I_2(m), m$ odd & $\ell>2$ & $I_2(\ell^{\nu_\ell(m)})$, $\ell$-unsuitable if $m\neq \ell^{\nu_\ell(m)}$   \\ \hline
					
					$I_2(m), m$ even 	& $\ell>2$ & $I_2(\ell^{\nu_\ell(m)})$ and $\tilde{I_2}(\ell^{\nu_\ell(m)})$, $\ell$-unsuitable if $m\neq 2\ell^{\nu_\ell(m)}$  \\ \hline
					
					$H_3$ & $2$ & $A_1^3$  \\ \hline
					
					$H_4$ & $2$ & $D_4$  	\\ \cline{2-3} 
					& $3$ & $A_2^2$	\\ \cline{2-3} 
					& $5$ & $I_2(5)^2$   \\ \hline
					
					$F_4$ & $2$  & $B_4$ and $C_4$  \\ \cline{2-3} 
					& $3$ & $A_2\times \tilde{A_2}$ \\ \hline
					
					$E_7$ & $2$ & $A_1\times D_6$ \\ \hline
					
					$E_8$ & $2$ & $D_8$  \\ \cline{2-3} 
					& $3$ & $A_2\times E_6$  \\ \cline{2-3}
					& $5$ &  $A_4^2$ \\ \hline
			\end{tabular}}
			\label{table:nonsupercuspidalreal}
		\end{center}
	\end{table}
	
	The non-dihedral exceptional cases of $W$ listed in Table \ref{table:nonsupercuspidalreal} are confirmed to satisfy property (\ref{suitability}) via straightforward MAGMA calculations.
	
	\begin{proposition}\label{prop:suitabilityofdihedrals}
		Let $W$ be of type $I_2(m)$ with $m\geq 5$. Then $W$ is $\ell$-suitable except when $\ell>2$ and
		\begin{itemize}
			\item[(i)] $m$ is odd and $m\neq \ell^{\nu_\ell(m)}$,
			\item[(ii)] $m$ is even and $m\neq2\ell^{\nu_\ell(m)}$.
		\end{itemize}
	\end{proposition}
	
	\begin{proof}
		If $W$ is not $\ell$-cuspidal then property (\ref{suitability}) holds by \cite[Observation 1.4]{METOO}. Since $W$ is only $\ell$-cuspidal when $\ell=2$ and $m$ is odd, so it remains to check $\ell$-suitability when $\ell=2$ and $m$ is even, or when $\ell>2$.
		
		If $\ell=2$ and $m$ is even, the type of an element of $\mathcal{R}_\ell(W)$ is $I_2(2^{\nu_2(m)})$. This is a Sylow $2$-subgroup, so $W$ is $2$-suitable.
		
		If $\ell>2$, the type of an element $R\in\mathcal{R}_\ell(W)$ is $I_2(\ell^{\nu_\ell(m)})$, with unique conjugacy class if $m$ is odd and two conjugacy classes if $m$ is even. The Sylow $\ell$-subgroup $S_\ell$ of $W$ is unique, so $N_W(S_\ell)=W$. If $m$ is a power of $\ell$, then $W$ is $\ell$-supercuspidal and $\ell$-suitability follows, so assume $m$ is not a power of $\ell$. If $m$ is odd, then $N_W(R)=R\lneq W$, so $W$ is $\ell$-unsuitable. If $m$ is even, then $N_W(R)$ is generated by $R$ and the order two rotation $(s_1s_2)^{m/2}\in W$, so $|N_W(R)|=2|R|$. Hence, $W$ is $\ell$-unsuitable unless $m=2\ell^{\nu_\ell(m)}$.
	\end{proof}
	
	As seen in Table \ref{table:nonsupercuspidalreal}, $W$ of type $B_n$ with $\ell=2$ is the only irreducible classical case that is $\ell$-cuspidal but not $\ell$-supercuspidal. Let $\text{Perm}(n)\cong \sigma_n$ be the group of $n\times n$ permutation matrices and $A(2,1,n)$ be the group of diagonal matrices with non-zero entries in $\{\pm 1\}$. Then $W(B_n)\cong A(2,1,n)\rtimes\text{Perm}(n)$ (see \cite[Chap. 2]{GUSTAY}). This notation will be generalized in Section \ref{complexcase}. 
	
	\begin{proposition}\label{prop:suitabilityofbn}
		If $W$ is of type $B_n$, then it is $2$-suitable.	
	\end{proposition}
	
	\begin{proof}
		Since $A(2,1,n)$ is a $2$-group, the Sylow $2$-subgroups of $A(2,1,n)\rtimes \text{Perm}(n)$ are given by $A(2,1,n)\rtimes S$, where $S\in \text{Syl}_2(\text{Perm}(n))$. By \cite[Corollary A.13.3]{Berkovich}, the Sylow $2$-subgroups in $\text{Perm}(n)$ are self-normalizing. Hence, $N_{A(2,1,n)\rtimes \text{Perm}(n)}(A(2,1,n)\rtimes S)=A(2,1,n)\rtimes S$, and we conclude that $W$ is $2$-suitable.
	\end{proof}
	
	\begin{corollary}\label{cor:suitability}
		Let $W$ be a finite real reflection group. Then $W$ is $\ell$-suitable except if $\ell>2$ and it has an irreducible component of type $I_2(m)$ with
		\begin{itemize}
			\item[(i)] $m$ odd and $m\neq \ell^{\nu_\ell(m)}$,
			\item[(ii)] $m$ even and $m\neq2\ell^{\nu_\ell(m)}$.
		\end{itemize}
	\end{corollary}
	
	\begin{proof}
		This follows from MAGMA calculations for the exceptional non-dihedral cases as well as Propositions \ref{prop:suitabilityofdihedrals} and \ref{prop:suitabilityofbn}.
	\end{proof}
	
	Corollary \ref{cor:suitability} shows that in most cases, the assumption of $\ell$-suitability in Theorem \ref{thm:normalizersylowref} is satisfied.
	
	\section{normalizers of Sylow subgroups in finite complex reflection groups}\label{complexcase}
	
	We now let $W$ be a finite complex reflection group acting on $V=\mathbb{C}^n$ for some $n\in \N$. The finite complex reflection groups were classified in \cite{TODD} and have Shephard-Todd numbering $G_1-G_{37}$ (see \cite[Chap. 8 \S 7]{GUSTAY}).
	
	\begin{example}\label{gmpn}
		Let $m,p,n$ be positive integers such that $p\mid m$. Then define \[A(m,p,n):=\{\text{diag}(\theta_1,\dots,\theta_n)\mid \theta_i\in \mathbb{C}, \theta_i^m=1 \ \text{and} \ (\prod_{i=1}^{n}\theta_i)^{m/p}=1 \}.\] Furthermore, let $\text{Perm}(n)$ be the $n\times n$ permutation matrices. Then $G(m,p,n):=A(m,p,n)\rtimes \text{Perm}(n)$ is a finite complex reflection group of order $\frac{m^n n!}{p}$ acting on $\mathbb{C}^n$. Although there is no canonical choice of generating sets like in finite real reflections there are standard minimal generating sets that are used throughout literature. Let $s_i=\mathbb{I}_{i-1}\oplus\text{adiag}(1,1)\oplus\mathbb{I}_{n-i-1}$ for $1\leq i \leq n-1$, $r=\mathbb{I}_{n-2}\oplus \text{adiag}(\zeta,\zeta^{-1})$ and $t=\mathbb{I}_{n-1}\oplus(\zeta)$ where $\zeta$ is a primitive $m$th root of unity. We list the standard generating sets often used for different cases of $G(m,p,n)$. \begin{align*}
			G(1,1,n)&=\langle s_1,s_2,\dots, s_{n-1} \rangle, \\
			G(m,1,n)&=\langle s_1,s_2,\dots,s_{n-1}, t\rangle, \\
			G(m,m,n)&=\langle s_1,s_2,\dots,s_{n-1},r\rangle, \\
			G(m,p,n)&=\langle s_1,s_2,\dots,s_{n-1}, r, t^p\rangle \hspace{1em} \text{for} \ p\neq 1,m, \\
			G(m,1,1)&=\langle t\rangle.
		\end{align*}
		For further details regarding these reflection groups see \cite[Chap. 2]{GUSTAY}
	\end{example}
	
	We will now introduce the preliminaries to prove Theorem \ref{thm:normalizersylowparacomp}. The $\ell$-Sylow classes of parabolic and reflection subgroups were classified in \cite{METOO}. The normalizer of a parabolic subgroup $P$ in a finite complex reflection group $W$ was described in \cite{MURA}. Similarly to \cite{HOWLETT}, it is proven that there exists a semidirect product decomposition $N_W(P)=P\rtimes U$, where the complement $U$ is often described as a stabilizer of a set of roots whose associated reflections generate $P$. We call the complement $U$ the \emph{MT-complement}. When $W$ is the complexification of a finite real reflection group, the MT-complement of a parabolic subgroup agrees with the H-complement. In some of the strictly complex cases an ad-hoc choice of roots must be made. Furthermore, in some cases there is no choice of roots for the generating reflections whose stabilizer is a complement (see \cite{MURA} for details). The MT-complement and other data can be calculated using the MAGMA code mentioned in \cite{MURA} and found at \url{https://www.maths.usyd.edu.au/u/don/software.html}. We note that the roots and generating reflections used by MAGMA for finite complex reflection groups correspond to the defining relations and diagrams given in \cite[App. 2, Tables 1-5]{Complexdiagrams} except in cases $G_{12}$ and $G_{22}$. Due to the ad-hoc nature of the roots stabilized by the MT-complement, relying on a generalization of Coxeter diagram automorphisms to prove Theorem \ref{thm:normalizersylowparacomp} appears to be unfeasible. Hence, we instead consider the irreducible $W$ that are not $\ell$-cuspidal, which are deduced from \cite[Table 1]{METOO}, and check if there is a Sylow $\ell$-subgroup normalized by the MT-complement.

	\begin{proof}[Proof of Theorem \ref{thm:normalizersylowparacomp}]
		By similar reasoning to Remark \ref{rem:reductiontoirred} it is sufficient to prove the theorem for each irreducible $W$. Consider the irreducible finite complex reflection groups $G_i$ for $i\in\{4,5,\dots,37\}$. We deduce the cases that are not $\ell$-cuspidal from \cite[Table 1]{METOO}. For each case, we use the MAGMA code from \cite{MURA} and add a function to find the Sylow $\ell$-subgroups of a $P_\ell\in\mathcal{P}_\ell(W)$ normalized by the MT-complement. In each case we find such a Sylow $\ell$-subgroup exists.
		
		Let us now consider the infinite family cases $G_1-G_3$. If $W=G_1= G(1,1,n)$, then $W$ is the complexification of the case of type $A_{n-1}$, so the MT-complement is the same as the H-complement, and have a Sylow $\ell$-subgroup normalized by this complement by Theorem \ref{thm:normalizersylowreal}. If $W=G_3= G(m,1,1)$ then $W$ is $\ell$-cuspidal and the result follows. Finally, if $W=G_2= G(m,p,n)$, $W$ is not $\ell$-cuspidal when $\ell\nmid m$. In particular, $P_\ell\in\mathcal{P}_\ell(W)$ is type $\prod_{i=0}^{k}G(1,1,\ell^i)^{b_i}$, where $n=(b_kb_{k-1}\dots b_1b_0)_\ell$. The MT-complement is described in \cite[Theorem 3.12 (iii)]{MURA} as a subgroup of $\prod_{i=0}^{k}G(m,1,b_i)$, where $\text{Perm}(b_i)$ permutes the components of type $G(1,1,\ell^i)$, while $A(m,1,b_i)$ multiplies the components of type $G(1,1,\ell^i)$ by an $m^{\text{th}}$ root of unity $\zeta_{j_i}$ for $1\leq j_i\leq b_i$ such that $\prod_{j_i=1}^{b_i}{\zeta_{j_i}}=1$. Now the matrix representatives of $A(m,1,b_i)$ commute with the Sylow $\ell$-subgroups of $G(1,1,\ell^i)^{b_i}$ and there exists a Sylow $\ell$-subgroup of $G(1,1,\ell^i)^{b_i}$ stable under $\text{Perm}(b_i)$ by the reasoning in Remark \ref{rem:reductiontoirred}. Hence, there is a Sylow $\ell$-subgroup of $P_\ell$ normalized by the MT-complement.
	\end{proof}
	
	\begin{remark}
		We do not investigate a generalization of Theorem \ref{thm:normalizersylowref} to finite complex reflection groups, since a splitting for normalizers of reflection subgroups does not always exist as noted in \cite[Example 6.1]{GHM3}.
	\end{remark}

	\section{Sylow subgroups stable under diagram automorphisms}\label{complexdiagramsproof}
	
	We will now investigate a generalization of Theorem \ref{thm:sylowstablediagramautroeal}. Finite complex reflection groups do not have standard notion of simple roots or simple reflections. However, there is a commonly used set of generating reflections, which can be found in \cite[App. 2, Tables 1-5]{Complexdiagrams}. For a fixed $W$, let $S$ be the set of these commonly used generating reflections. We define a \emph{diagram automorphism} of a complex reflection group $W$ as a permutation of $S$ that induces a group automorphism. These diagram automorphisms permute the diagrams given in \cite[App. 2, Tables 1-5]{Complexdiagrams}, generalizing the notion of Coxeter diagram automorphism. Proposition \ref{prop:diagramautomorphismscomp} is a classification of the non-trivial diagram automorphisms $\rho$ for irreducible $W$, where $\rho$ is written as a permutation of the generating reflections with respect to the ordering of the diagrams in Figure~\ref{fig:complexdiagram}.
	
	\begin{remark}
		A \emph{reflection coset} of $W$ is given by $\gamma W$, where $\gamma$ is a finite order element of $N_{\text{GL}(V)}(W)$. The reflection cosets of complex reflection groups have been studied in relation to twisted invariant theory \cite{twistedinvarianttheory,reflectionsubquotients}. A classification of the reflection cosets is given in \cite[3E]{SPETS1} or alternatively \cite[Table D.5]{GUSTAY}. In \cite{autocompref}, structural results regarding the automorphism group of a complex reflection group are proven. They define the diagram automorphisms to be the outer automorphisms induced from $N_{\text{GL}(V)}(W)$, since these automorphisms correspond to the diagram automorphisms of finite Coxeter groups. These diagram automorphisms differ to those we have defined above. In particular, our diagram automorphisms include elements that do not belong to $\text{GL}(V)$ and do not include any of the induced outer automorphisms that do not permute the generating set (compare \cite[3E]{SPETS1} and Proposition \ref{prop:diagramautomorphismscomp}). Alternative diagrams to those found in \cite[App. 2, Tables 1-5]{Complexdiagrams} have also been studied. For example, in \cite[2.2]{eer} an alternative diagram is given for $G(e,e,n)$ such that the  reflection cosets do indeed induce a permutation of the diagram.
	\end{remark}
	
	\begin{proposition}\label{prop:diagramautomorphismscomp}
		Let $W$ be an irreducible finite complex reflection group. Then the non-trivial diagram automorphisms of $W$ with respect to the labelling in Figure \ref{fig:complexdiagram} are:
		\begin{itemize}
			\item[(i)] $G_1=G(1,1,n)$ with $\rho=(1\;n-1)(2\;n-2)\dots (\lceil\frac{n-1}{2}\rceil\;\lceil\frac{n}{2}\rceil)$.
			\item[(ii)]  $G(m,p,n)$ and $G(m,m,n)$, with $\rho=(n-1\; n)$.
			\item[(iii)]  $G(4,2,2), G_{12}$ and $G_{22}$ with $\rho\in\langle (1\;2),(1\;2\;3)\rangle$.
			\item[(iv)] $G(2,2,4)$ with $\rho\in \langle(1\; 3),(1\;3\;4)\rangle$.
			\item[(v)] $G_k$ for $k=4,5,8,16,20,24$ with $\rho=(1\; 2)$.
			\item[(vi)] $G_7$ with $\rho=(2\;3)$.
			\item[(vii)] $G_{25}$ with  $\rho=(1\;3)$.
			\item[(viii)] $G_{k}$ for $k=28,32$ with $\rho=(1\;4)(2\;3)$.
			\item[(ix)] $G_{k}$ for $k=31,33$ with $\rho=(1\;5)(2\;4)$.
			\item[(x)] $G_{35}$ with $\rho=(1\;6)(3\;5)$.
		\end{itemize}
	\end{proposition}
	
	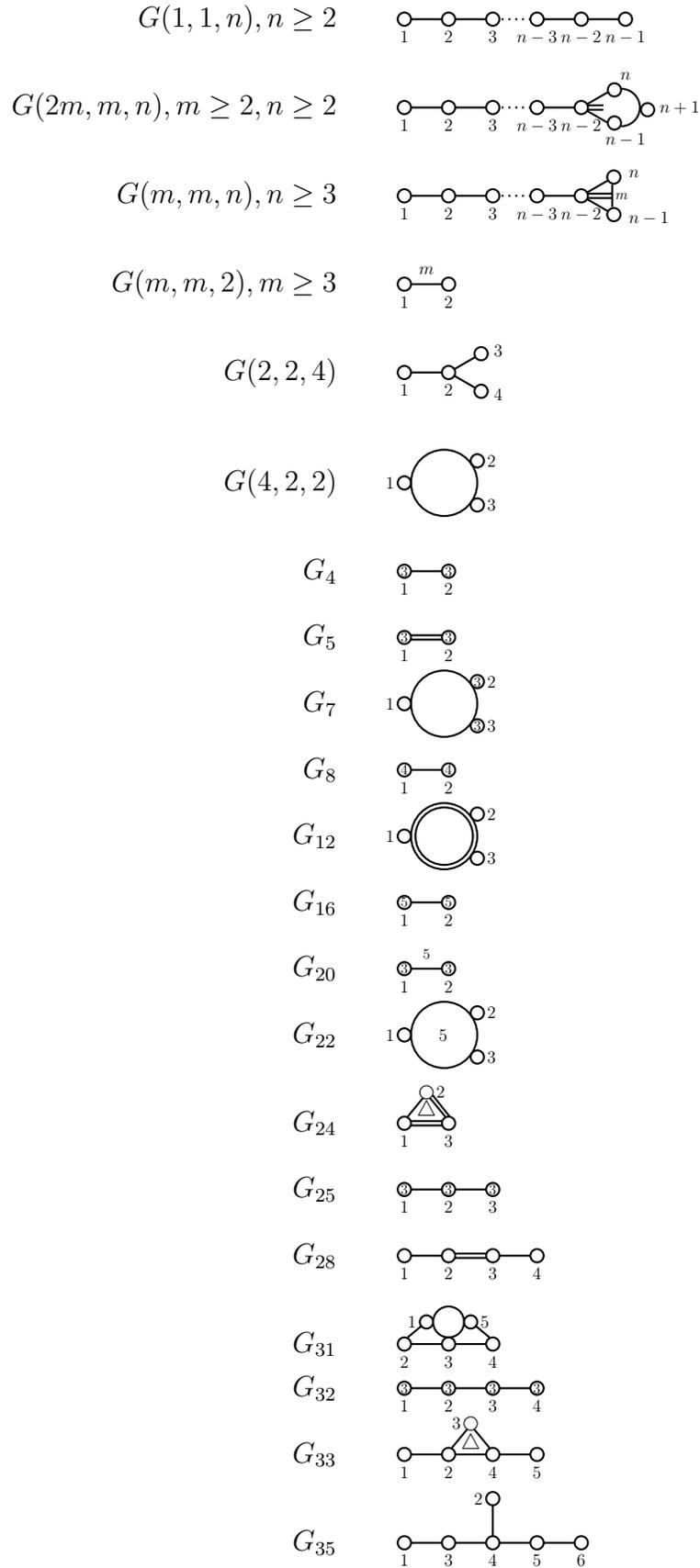
\begin{figure}[]
		\caption{Complex reflection groups with diagram automorphisms}
		\begin{center}
			\begin{tikzpicture}[scale=.32]
				
				\draw (-2.6,0) node[anchor=east]  {$G(1,1,n), n\geq 2$};
				\foreach \x in {0,...,5}
				\draw[xshift=\x cm,thick] (\x cm,0) circle (.3cm);
				\draw[dotted,thick] (4.4 cm,0) -- +(1.4 cm,0);
				\foreach \y in {0.15,1.15,3.15,4.15}
				\draw[xshift=\y cm,thick] (\y cm,0) -- +(1.4 cm,0);
				\draw (0,-0.1) node[below] {\scalebox{0.6}{1}};
				\draw (2,-0.1) node[below] {\scalebox{0.6}{2}};
				\draw (4,-0.1) node[below] {\scalebox{0.6}{3}};
				\draw (6,-0.1) node[below] {\scalebox{0.6}{$n-3$}};
				\draw (8,-0.1) node[below] {\scalebox{0.6}{$n-2$}};
				\draw (10,-0.1) node[below] {\scalebox{0.6}{$n-1$}};
				
				\draw (-2.6,-4) node[anchor=east]  {$G(2m,m,n), m\geq 2, n\geq 2$};
				\foreach \x in {0,...,4}
				\draw[xshift=\x cm,thick] (\x cm,-4) circle (.3cm);
				\draw[dotted,thick] (4.4 cm,-4) -- +(1.4 cm,0);
				\foreach \y in {0.15,1.15,3.15}
				\draw[xshift=\y cm,thick] (\y cm,-4) -- +(1.4 cm,0);
				\draw[xshift=8 cm,yshift=-4 cm,thick] (30: 3 mm) -- (30: 14 mm);
				\draw[xshift=8 cm,yshift=-4 cm,thick] (-30: 3 mm) -- (-30: 14 mm);	
				\draw (0,-4.1) node[below] {\scalebox{0.6}{1}};
				\draw (2,-4.1) node[below] {\scalebox{0.6}{2}};
				\draw (4,-4.1) node[below] {\scalebox{0.6}{3}};
				\draw (6,-4.1) node[below] {\scalebox{0.6}{$n-3$}};
				\draw (8,-4.1) node[below] {\scalebox{0.6}{$n-2$}};
				\draw (10,-4.7) node[below] {\scalebox{0.6}{$n-1$}};
				\draw (10,-3.2) node[above] {\scalebox{0.6}{$n$}};
				\draw[thick] (8.3,-4.1) --+ (0.7,0);
				\draw[thick] (8.3,-3.9) --+ (0.7,0);
				\draw[thick] ( 9.5 cm,-3.2) circle (.3cm);
				\draw[thick] ( 9.5 cm,-4.7) circle (.3cm);
				\draw[thick] ( 11 cm,-4.1) circle (.3cm);
				\draw[thick] (9.8,-3.15) arc (90:-90:0.85);
				\draw (11.1,-4) node[right] {\scalebox{0.6}{$n+1$}};
				
				\draw (-2.6,-8) node[anchor=east]  {$G(m,m,n), n\geq3$};
				\foreach \x in {0,1,...,4}
				\draw[xshift=\x cm,thick] (\x cm,-8) circle (.3cm);
				\draw[xshift=8 cm,yshift=-8 cm,thick] (30: 17 mm) circle (.3cm);
				\draw[xshift=8 cm,yshift=-8 cm,thick] (-30: 17 mm) circle (.3cm);
				\draw[dotted,thick] (4.4 cm,-8) -- +(1.4 cm,0);
				\foreach \y in {0.15,1.15,3.15}
				\draw[xshift=\y cm,thick] (\y cm,-8) -- +(1.4 cm,0);
				\draw[xshift=8 cm,yshift=-8 cm,thick] (30: 3 mm) -- (30: 14 mm);
				\draw[xshift=8 cm,yshift=-8 cm,thick] (-30: 3 mm) -- (-30: 14 mm);	
				\draw (0,-8.1) node[below] {\scalebox{0.6}{1}};
				\draw (2,-8.1) node[below] {\scalebox{0.6}{2}};
				\draw (4,-8.1) node[below] {\scalebox{0.6}{3}};
				\draw (6,-8.1) node[below] {\scalebox{0.6}{$n-3$}};
				\draw (8,-8.1) node[below] {\scalebox{0.6}{$n-2$}};
				\draw (9.7,-7) node[right] {\scalebox{0.6}{$n$}};
				\draw (9.7,-9) node[right] {\scalebox{0.6}{$n-1$}};
				\draw[thick] (8.3,-8.1) --+ (1.1,0);
				\draw[thick] (8.3,-7.9) --+ (1.1,0);
				\draw[thick] (9.4,-7.5) --+ (0,-1);
				\draw (9.1,-8) node[right] {\scalebox{0.5}{$m$}};
				
				\draw (-2.6,-12) node[anchor=east]  {$G(m,m,2), m\geq 3$};
				\foreach \x in {0,1}
				\draw[xshift=\x cm,thick] (\x cm,-12) circle (.3cm);
				\draw[thick] (0.3,-12) -- +(1.41 cm,0);
				\draw (0,-12.1) node[below] {\scalebox{0.6}{1}};
				\draw (2,-12.1) node[below] {\scalebox{0.6}{2}};
				\draw (1,-12) node[above] {\scalebox{0.6}{$m$}};
				
				\draw (-2.6,-16) node[anchor=east]  {$G(2,2,4)$};
				\foreach \x in {0,1}
				\draw[xshift=\x cm,thick] (\x cm,-16) circle (.3cm);
				\draw[xshift=2 cm,yshift=-16 cm,thick] (30: 17 mm) circle (.3cm);
				\draw[xshift=2 cm,yshift=-16 cm,thick] (-30: 17 mm) circle (.3cm);
				\draw[thick] (0.3 cm,-16) -- +(1.4 cm,0);
				\draw[xshift=2 cm,yshift=-16 cm,thick] (30: 3 mm) -- (30: 14 mm);
				\draw[xshift=2 cm,yshift=-16 cm,thick] (-30: 3 mm) -- (-30: 14 mm);	
				\draw (0,-16.1) node[below] {\scalebox{0.6}{1}};
				\draw (2,-16.1) node[below] {\scalebox{0.6}{2}};
				\draw (3.6,-15) node[right] {\scalebox{0.6}{3}};
				\draw (3.6,-17) node[right] {\scalebox{0.6}{4}};

				\draw (-2.6,-21) node[anchor=east]  {$G(4,2,2)$};
				\draw[thick] (0 cm,-21) circle (.3cm);
				\draw[thick] (1.8 cm,-21) circle (1.5cm);
				\draw[thick] (3.3 cm,-22) circle (.3cm);
				\draw[thick] (3.3 cm,-20) circle (.3cm);
				\draw (3.3,-22) node[right] {\scalebox{0.6}{3}};
				\draw (3.3,-20) node[right] {\scalebox{0.6}{2}};
				\draw (0,-21) node[left] {\scalebox{0.6}{1}};
				
				\draw (-2.6,-25) node[anchor=east]  {$G_4$};
				\foreach \x in {0,1}
				\draw[xshift=\x cm,thick] (\x cm,-25) circle (.3cm);
				\draw[thick] (0.3,-25) -- +(1.41 cm,0);
				\draw (0,-25.1) node[below] {\scalebox{0.6}{1}};
				\draw (2,-25.1) node[below] {\scalebox{0.6}{2}};
				\draw (0,-25) node {\scalebox{0.5}{3}};
				\draw (2,-25) node {\scalebox{0.5}{3}};
				
				\draw (-2.6,-28) node[anchor=east]  {$G_5$};
				\foreach \x in {0,1}
				\draw[xshift=\x cm,thick] (\x cm,-28) circle (.3cm);
				\draw[thick] (0.3 cm,-28.1) -- +(1.41 cm,0);
				\draw[thick] (0.3 cm,-27.9) -- +(1.41 cm,0);
				\draw (0,-28.1) node[below] {\scalebox{0.6}{1}};
				\draw (2,-28.1) node[below] {\scalebox{0.6}{2}};
				\draw (0,-28) node {\scalebox{0.5}{3}};
				\draw (2,-28) node {\scalebox{0.5}{3}};
				
				\draw (-2.6,-31) node[anchor=east]  {$G_{7}$};
				\draw[thick] (0 cm,-31) circle (.3cm);
				\draw[thick] (1.8 cm,-31) circle (1.5cm);
				\draw[thick] (3.3 cm,-32) circle (.3cm);
				\draw[thick] (3.3 cm,-30) circle (.3cm);
				\draw (3.3,-30) node {\scalebox{0.5}{3}};
				\draw (3.3,-32) node {\scalebox{0.5}{3}};
				\draw (3.3,-32) node[right] {\scalebox{0.6}{3}};
				\draw (3.3,-30) node[right] {\scalebox{0.6}{2}};
				\draw (0,-31) node[left] {\scalebox{0.6}{1}};

				\draw (-2.6,-34) node[anchor=east]  {$G_8$};
				\foreach \x in {0,1}
				\draw[xshift=\x cm,thick] (\x cm,-34) circle (.3cm);
				\draw[thick] (0.3,-34) -- +(1.41 cm,0);
				\draw (0,-34.1) node[below] {\scalebox{0.6}{1}};
				\draw (2,-34.1) node[below] {\scalebox{0.6}{2}};
				\draw (0,-34) node {\scalebox{0.5}{4}};
				\draw (2,-34) node {\scalebox{0.5}{4}};
				
						\draw (-2.6,-37) node[anchor=east]  {$G_{12}$};
				\draw[thick] (0 cm,-37) circle (.3cm);
				\draw[thick] (1.8 cm,-37) circle (1.5cm);
					\draw[thick] (1.8 cm,-37) circle (1.3cm);
				\draw[thick] (3.3 cm,-38) circle (.3cm);
				\draw[thick] (3.3 cm,-36) circle (.3cm);
				\draw (3.3,-38) node[right] {\scalebox{0.6}{3}};
				\draw (3.3,-36) node[right] {\scalebox{0.6}{2}};
				\draw (0,-37) node[left] {\scalebox{0.6}{1}};

				\draw (-2.6,-40) node[anchor=east]  {$G_{16}$};
				\foreach \x in {0,1}
				\draw[xshift=\x cm,thick] (\x cm,-40) circle (.3cm);
				\draw[thick] (0.3,-40) -- +(1.41 cm,0);
				\draw (0,-40.1) node[below] {\scalebox{0.6}{1}};
				\draw (2,-40.1) node[below] {\scalebox{0.6}{2}};
				\draw (0,-40) node {\scalebox{0.5}{5}};
				\draw (2,-40) node {\scalebox{0.5}{5}};

				\draw (-2.6,-43) node[anchor=east]  {$G_{20}$};
				\foreach \x in {0,1}
				\draw[xshift=\x cm,thick] (\x cm,-43) circle (.3cm);
				\draw[thick] (0.3,-43) -- +(1.41 cm,0);
				\draw (0,-43.1) node[below] {\scalebox{0.6}{1}};
				\draw (2,-43.1) node[below] {\scalebox{0.6}{2}};
				\draw (0,-43) node {\scalebox{0.5}{3}};
				\draw (2,-43) node {\scalebox{0.5}{3}};
				\draw (1,-43) node[above] {\scalebox{0.5}{$5$}};

						\draw (-2.6,-46) node[anchor=east]  {$G_{22}$};
				\draw[thick] (0 cm,-46) circle (.3cm);
				\draw[thick] (1.8 cm,-46) circle (1.5cm);
				\draw[thick] (3.3 cm,-47) circle (.3cm);
				\draw[thick] (3.3 cm,-45) circle (.3cm);
				\draw (3.3,-47) node[right] {\scalebox{0.6}{3}};
				\draw (3.3,-45) node[right] {\scalebox{0.6}{2}};
				\draw (0,-46) node[left] {\scalebox{0.6}{1}};
					\draw (1.75,-46) node {\scalebox{0.6}{5}};
				
				\draw (-2.6,-50) node[anchor=east]  {$G_{24}$};
				\foreach \x in {0,1}
				\draw[xshift=\x cm,thick] (\x cm,-50) circle (.3cm);
				\draw[thick] (0.3 cm,-50.1) --  +(1.41 cm,0);
				\draw[thick] (0.3 cm,-49.9) --  +(1.41 cm,0);
				\draw (0,-50.1) node[below] {\scalebox{0.6}{1}};
				\draw (2,-50.1) node[below] {\scalebox{0.6}{3}};
				\draw (1,-48.6) circle (.3cm);
				\draw[thick] (0.15 cm,-49.7) --  +(0.7cm,0.85);
				\draw[thick] (1.85 cm,-49.7) --  +(-0.7cm,0.85);
					\draw[thick] (2.05 cm,-49.7) --  +(-0.77cm,0.97);
					\draw (1,-50.3) node[above] {\scalebox{0.75}{$\triangle$}};
				\draw (1,-48.6) node[right] {\scalebox{0.6}{2}};
				
				\draw (-2.6,-53) node[anchor=east]  {$G_{25}$};
				\foreach \x in {0,1,2}
				\draw[xshift=\x cm,thick] (\x cm,-53) circle (.3cm);
				\draw[thick] (0.3,-53) -- +(1.41 cm,0);
				\draw[thick] (2.3,-53) -- +(1.41 cm,0);
				\draw (0,-53.1) node[below] {\scalebox{0.6}{1}};
				\draw (2,-53.1) node[below] {\scalebox{0.6}{2}};
				\draw (4,-53.1) node[below] {\scalebox{0.6}{3}};
				\draw (0,-53) node {\scalebox{0.5}{3}};
				\draw (2,-53) node {\scalebox{0.5}{3}};
				\draw (4,-53) node  {\scalebox{0.5}{3}};
				
				\draw (-2.6,-56) node[anchor=east]  {$G_{28}$};
				\foreach \x in {0,...,3}
				\draw[xshift=\x cm,thick] (\x cm,-56) circle (.3cm);
				\foreach \y in {0.15,2.15}
				\draw[xshift=\y cm,thick] (\y cm,-56) -- +(1.4 cm,0);
				\draw[thick] (2.3 cm,-56.1) --  +(1.41 cm,0);
				\draw[thick] (2.3 cm,-55.9) --  +(1.41 cm,0);
				\draw (0,-56.1) node[below] {\scalebox{0.6}{1}};
				\draw (2,-56.1) node[below] {\scalebox{0.6}{2}};
				\draw (4,-56.1) node[below] {\scalebox{0.6}{3}};
				\draw (6,-56.1) node[below] {\scalebox{0.6}{4}};
				
				\draw (-2.6,-60) node[anchor=east]  {$G_{31}$};
				\draw[thick] ( 0cm,-60) circle (.3cm);
				\draw[thick] (2 cm,-60) circle (.3cm);
				\draw[thick] (4 cm,-60) circle (.3cm);
				\draw[thick] (1 cm,-59) circle (.3cm);
				\draw[thick] (3 cm,-59) circle (.3cm);
				\draw[thick] ( 2 cm,-59) circle (0.7cm);
				\foreach \y in {0.15,1.15}
				\draw[xshift=\y cm,thick] (\y cm,-60) -- +(1.4 cm,0);
				\draw[thick] (0.15 cm,-59.7) --  +(0.6cm,0.5);
				\draw[thick] (3.85 cm,-59.7) --  +(-0.6cm,0.5);
				\draw (0,-60.1) node[below] {\scalebox{0.6}{2}};
				\draw (2,-60.1) node[below] {\scalebox{0.6}{3}};
				\draw (4,-60.1) node[below] {\scalebox{0.6}{4}};
				\draw (1,-59) node[left] {\scalebox{0.6}{1}};
				\draw (3,-59) node[right] {\scalebox{0.6}{5}};

				\draw (-2.6,-62) node[anchor=east]  {$G_{32}$};
				\foreach \x in {0,...,3}
				\draw[xshift=\x cm,thick] (\x cm,-62) circle (.3cm);
				\foreach \y in {1.15,0.15,2.15}
				\draw[xshift=\y cm,thick] (\y cm,-62) -- +(1.4 cm,0);
				\draw (0,-62.1) node[below] {\scalebox{0.6}{1}};
				\draw (2,-62.1) node[below] {\scalebox{0.6}{2}};
				\draw (4,-62.1) node[below] {\scalebox{0.6}{3}};
				\draw (6,-62.1) node[below] {\scalebox{0.6}{4}};
				\draw (0,-62) node {\scalebox{0.5}{3}};
				\draw (2,-62) node {\scalebox{0.5}{3}};
				\draw (4,-62) node {\scalebox{0.5}{3}};
				\draw (6,-62) node {\scalebox{0.5}{3}};

				\draw (-2.6,-65) node[anchor=east]  {$G_{33}$};
				\foreach \x in {0,...,3}
				\draw[xshift=\x cm,thick] (\x cm,-65) circle (.3cm);
				\foreach \y in {1.15,0.15,2.15}
				\draw[xshift=\y cm,thick] (\y cm,-65) -- +(1.4 cm,0);
				\draw (0,-65.1) node[below] {\scalebox{0.6}{1}};
				\draw (2,-65.1) node[below] {\scalebox{0.6}{2}};
				\draw (4,-65.1) node[below] {\scalebox{0.6}{4}};
				\draw (6,-65.1) node[below] {\scalebox{0.6}{5}};
				\draw[thick] (2.15 cm,-64.7) --  +(0.7cm,0.85);
				\draw[thick] (3.85 cm,-64.7) --  +(-0.7cm,0.85);
				\draw (3 cm,-63.6) circle (.3cm);
				\draw (3,-63.6) node[left] {\scalebox{0.6}{3}};
				\draw (3,-65.4) node[above] {\scalebox{0.75}{$\triangle$}};
				
				\draw (-2.6,-69) node[anchor=east]  {$G_{35}$};
				\foreach \x in {0,...,4}
				\draw[thick,xshift=\x cm] (\x cm,-69) circle (3 mm);
				\foreach \y in {0,...,3}
				\draw[thick,xshift=\y cm] (\y cm,-69) ++(.3 cm, 0) -- +(14 mm,0);
				\draw[thick] (4 cm,-67 cm) circle (3 mm);
				\draw[thick] (4 cm, -68.7 cm) -- +(0, 1.4 cm);
				\draw (2,-69.1) node[below] {\scalebox{0.6}{3}};
				\draw (4,-69.1) node[below] {\scalebox{0.6}{4}};
				\draw (6,-69.1) node[below] {\scalebox{0.6}{5}};
				\draw (8,-69.1) node[below] {\scalebox{0.6}{6}};
				\draw (0,-69.1) node[below] {\scalebox{0.6}{1}};
				\draw (4,-67) node[left] {\scalebox{0.6}{2}};
			\end{tikzpicture}
		\end{center}
		\label{fig:complexdiagram}
	\end{figure}
	
	This classification follows from  inspection of the diagrams and relations in \cite[App.2, Tables 1-5]{Complexdiagrams}. We provide MAGMA calculations to confirm the diagram automorphisms are indeed automorphisms in the exceptional cases. We give an example for the cases arising from $G(m,p,n)$ with $m,n\geq 2$, $p>2$ and $m\neq p$, as this requires some work.

	\begin{example}
		Let $W=G(m,p,n)$ with $m,n\geq 2$, $p>2$ and $m\neq p$. The relations on the set of generators $s_1,s_2,\dots ,s_{n+1}$ are \begin{align*}
			(1)& \hspace{5mm} s_1^2=s_2^2=\dots=s_n^2=s_{n+1}^{m/p} =1, \\
			(2)& \hspace{5mm} s_{i}s_j=s_js_i \ \text{for} \ 1\leq i,j\leq n-2 \ \text{and} \ |i-j|\geq2, \\
			(3)& \hspace{5mm} s_{i}s_{i+1}s_{i}=s_{i+1}s_is_{i+1} \ \text{for} \ 1\leq i\leq n-3, \\
			(4)& \hspace{5mm} s_{n+1}s_{n-2} =s_{n-2}s_{n+1}, \\
			(5)& \hspace{5mm} s_{n-1}s_{n-2}s_{n-1}=s_{n-2}s_{n-1}s_{n-2}, \\
			(6)& \hspace{5mm} s_{n}s_{n-2}s_{n}=s_{n-2}s_{n}s_{n-2}, \\
			(7)& \hspace{5mm} \underbrace{s_{n}s_{n+1}s_{n-1}s_{n}s_{n-1}s_{n}\dots}_{p+1 \ \text{factors}}=\underbrace{s_{n+1}s_{n-1}s_{n}s_{n-1}s_{n}\dots}_{p+1 \ \text{factors}}.
		\end{align*}
		
		It is clear from relations (1)-(6) that the only candidate for a non-trivial diagram automorphism is fixing $s_i$ for $1\leq i \leq n-2$ and $i=n+1$, as well as switching $s_{n-1}$ and $s_n$. For example, relations (1)-(4) are fixed and relations (5) and (6) are switched by this permutation on the generators. Call this permutation $\rho$ and consider its action on relation (7). We will use the matrix representatives to show that this relation under $\rho$ is only satisfied when $m=2p$. Choose the standard matrix generators for $G(m,p,n)$ seen in Example~\ref{gmpn}, where $s_n=r$ and $s_{n+1}=t^p$. After applying $\rho$ to relation (7) and multiplying the matrix representatives, we find that that $\rho$ is a diagram automorphism if and only if $\zeta^{2p}=1$. Hence, the only non-trivial diagram automorphism of $G(m,p,n)$ is $\rho$ when $m=2p$.
	\end{example}

	\begin{theorem}\label{thm:sylowstablediagramautocomp}
		Let $W$ be an $\ell$-cuspidal finite complex reflection group. There exists a Sylow $\ell$-subgroup of $W$ stable under the diagram automorphisms of $W$, except when $W$ has an irreducible component of type:
		\begin{itemize}
			\item[(a)] $G(2m,m,n)$ when $\ell>2$ and $\ell\mid m,n$.
			\item[(b)] $G(m,m,n)$ when $\ell>2$, $\ell\mid m,n$ and $m$ is not a power of $\ell$.
			\item[(c)] $G_{12}$ when $\ell=2$.
			\item[(d)] $G_{22}$ when $\ell=2,5$.
			\item[(e)] $G_{28}=F_4$ when $\ell=3$.
			\item[(f)]$G_{31}=O_4$ when $\ell=3,5$.
		\end{itemize}
		In particular, if $W$ is $\ell$-supercuspidal, then there exists a Sylow $\ell$-subgroup of $W$ stable under the diagram automorphisms of $W$, except when $W$ has an irreducible component of type $G_{12}$ when $\ell=2$.
	\end{theorem}
	
	\begin{proof}
		Of the cases (i)-(x) in Proposition \ref{prop:diagramautomorphismscomp}, the $\ell$-cuspidal cases are deduced using \cite[Table 1]{METOO}. The $\ell$-cuspidal exceptional cases are recorded in Table \ref{table:sylowcomplexnormal}, with MAGMA being used to deduce the number of Sylow $\ell$-subgroups stable under $\rho$. In Table \ref{table:sylowcomplexnormal}, $\text{Syl}_\ell^\rho(W)$ is the set of Sylow $\ell$-subgroups stable under $\rho$.
		
		\begin{table}[h!]
			\begin{center}
				\caption{Sylow $\ell$-subgroups stable under $\rho$ for cuspidal exceptional $W$}
				\scalebox{1}{\begin{tabular}{ | c | c | c | c| c| c|}
						\hline
						{$W$} & $\rho$ & $\ell$ & $\ell$-supercuspidal &  $|\text{Syl}_\ell((W)|$ & $|\text{Syl}_\ell^\rho(W)|$   \\ \hline \hline
						$G_4$ & $(1 \; 2)$ & $2$ & \checkmark & $1$ & $1$ \\ \hline	
						
						$G_5$ & ($1\;2)$ & $2$ & \xmark & $1$ & $1$ \\ \hline	
						
						$G_7$ & $(2\;3)$ & $2$ & \xmark & $1$ & $1$ \\ \cline{3-6}
						& & $3$ & \xmark & $4$ & $2$ \\ \hline	
						
						$G_8$ & $(1\;2)$ & $2$ & \xmark & $3$ & $1$ \\ \cline{3-6}
						& & $3$ & \checkmark & $4$ & $2$ \\ \hline	
						
							$G_{12}$ &  $\langle (1\;2),(1\;2\;3)\rangle$ & $2$ & \checkmark & $3$ & $0$ \\ \cline{3-6}
						& & $3$ & \xmark & $4$ & $1$ \\ \hline	
						
						$G_{16}$ & $(1\;2)$ & $2$ & \checkmark & $5$ & $1$ \\ \cline{3-6}
						& & $3$ & \checkmark & $10$ & $2$ \\ \cline{3-6}
						& & $5$ & \xmark & $6$ & $2$ \\ \hline
						
						$G_{20}$ & $(1\;2)$ & $2$ & \xmark & $5$ & $1$ \\ \cline{3-6}
						& & $3$ & \xmark & $10$ & $2$ \\ \cline{3-6}
						& & $5$ & \checkmark & $6$ & $2$ \\ \hline
						
							$G_{22}$ & $\langle (1\;2),(1\;2\;3)\rangle$ & $2$ & \xmark & $5$ & $0$ \\ \cline{3-6}
						& & $3$ & \xmark & $10$ & $1$ \\ \cline{3-6}
						& & $5$ & \xmark & $6$ & $0$ \\ \hline
						
						$G_{24}=J_3^{(4)}$ & $(1\;2)$ & $7$ & \checkmark &  $8$ & $2$ \\ \hline
						
						$G_{25}=L_3$ & $(1\;3)$ & $3$ & \checkmark & $4$ & $4$ \\ \hline

						$G_{28}=F_4$  & $(1\;4)(2\;3)$ & $2$ & \xmark & $9$ & $3$ \\ \cline{3-6}
						& & $3$ & \xmark & $16$ & $0$ \\ \hline
						
						$G_{31}=O_4$  & $(1\;5)(2\;4)$ & $2$ & \xmark & $45$ & $9$ \\ \cline{3-6}
						& & $3$ & \xmark & $160$ & $0$ \\ \cline{3-6}
						& & $5$ & \xmark &  $576$ & $0$ \\ \hline
						
						$G_{32}=L_4$ & $(1\;4)(2\;3)$ & $2$ & \checkmark & $135$ & $11$ \\ \cline{3-6}
						& & $5$ & \checkmark & $1296$ & $24$ \\ \hline
						
						$G_{33}=K_5$ & $(1\;5)(2\;4)$ & $2$ &  \xmark & $135$ & $9$  \\ \hline
						
						$G_{35}=E_6$ & $(1\;6)(3\;5)$ & $3$ & \checkmark & $160$ &  $32$ \\ \hline
				\end{tabular}}
				\label{table:sylowcomplexnormal}
			\end{center}
		\end{table}
		
		The cases $G(1,1,n)$ and $G(2,2,4)$ reduce to the real reflection group setting. The case $G(4,2,2)$ is trivial since it is a $2$-group. The remaining $\ell$-cuspidal cases are $G(xm,m,n)$ where $x\in\{1,2\}$ and $\ell\mid m$. If $\ell=2$, then by Lemma \ref{lem:groupautosylow} there is a Sylow $2$-subgroup stabilized by $\rho$. Now assume $\ell>2$ and consider $G(xm,m,n)=A(xm,m,n)\rtimes \text{Perm}(n)$. We note that $G(2m,m,n)$ has a reflection subgroup of type $G(m,m,n)$, the order two diagram automorphism of $G(2m,m,n)$ induces the order two diagram automorphism of $G(m,m,n)$ and $\nu_\ell(|G(2m,m,n)|)=\nu_\ell(|G(m,m,n)|)$. Hence, if we prove part (b), then (a) follows. Let $\rho$ be the diagram automorphism of $G(m,m,n)=\langle s_1,\dots,s_{n-1},r\rangle$ switching $s_{n-1}$ and $r$, while fixing $s_1,\dots,s_{n-2}$, with generators from Example \ref{gmpn}.
		
		If $\ell\nmid n$, then $A(\ell^{\nu_\ell(m)},\ell^{\nu_\ell(m)},n)\rtimes S_\ell$ with where $S_\ell\in \text{Syl}(\langle s_1,\dots,s_{n-2}\rangle)$ is a Sylow $\ell$-subgroup of $G(m,m,n)$ and is fixed by $\rho$.
		
		Now let $\ell\mid n$. Firstly, if $m$ is a power of $\ell$, then $A(m,m,n)\rtimes S_\ell$ with $S_\ell\in \text{Syl}_\ell(\text{Perm}(n))$ is a Sylow $\ell$-subgroup of $G(m,m,n)$ stable under $\rho$. Now assume $m$ is not a power of $\ell$. Suppose there is a Sylow $\ell$-subgroup $S_\ell$ of $G(m,m,n)$ stable under $\rho$. Since $G(m,m,n)$ is $\ell$-cuspidal, there is an element of $S\in S_\ell$, with order some power of $\ell$ that does not fix $\mathbb{C}e_n$ setwise. The element $S$ can be written as $AP$ where $A\in A(m,m,n)$ and $P\in\text{Perm}(n)$. Note that $P$ has non-zero terms in the matrix positions $(n,i)$ and $(j,n)$ for some $1\leq i\neq j<n$, because if $i=j$ the order of $P$ is even, contradicting that $S=AP$ has order a power of $\ell$. Now $\rho(P)P^{-1}$ is a diagonal matrix consisting of nonzero entries $\zeta, \zeta^{-1}$ and $1$. Hence, $\rho(P)P^{-1}$ has order $m$. Furthermore, it is clear that $\rho(A(m,m,n))=A(m,m,n)$ since a word in $s_1,\dots,s_{n-1},r$ that gives a diagonal matrix will still be diagonal after switching $s_{n-1}$ and $r$ in the word. Hence, $\rho(S)S^{-1}=\rho(AP)P^{-1}A^{-1}=\rho(A)\rho(P)P^{-1}A^{-1}\in S_\ell$ where $\rho(A), \rho(P)P^{-1}, A^{-1}\in A(m,m,n)$. Since these elements are all diagonal, they commute, so $|\rho(S)S^{-1}|=\text{lcm}(|\rho(A)|, |\rho(P)P^{-1}|,|A^{-1}|)$, which is a multiple of $m$,  so is not a power of $\ell$. This contradicts $\rho(S)S^{-1}\in S_\ell$, so no Sylow $\ell$-subgroup of $G(m,m,n)$ is stable under $\rho$ when $m$ is not a power of $\ell$.
		
		This completes the proof of when $W$ is $\ell$-cuspidal. Since none of the exceptions listed are $\ell$-supercuspidal except $G_{12}$ when $\ell=2$ (see \cite[Table 3]{METOO}), the result follows.
	\end{proof}
	
	\begin{remark}
		We observe a connection between the exceptions found in (e) and (f) of Theorem~\ref{thm:sylowstablediagramautocomp}. If $G_{31}=\langle s_1,\dots, s_5 \rangle$ with respect to the ordering in Figure \ref{fig:complexdiagram}, then \[R:=\langle s_1,s_2,s_4,s_5\mid (s_1s_5)^4=(s_1s_2)^3=(s_1s_4)^2=(s_2s_4)^2=(s_2s_5)^3=1\rangle\] is isomorphic to the reflection group of type $F_4$. Furthermore, the order $2$ diagram automorphism of $G_{31}$ given by $\rho=(1\;5)(2\;4)$ induces the order $2$ diagram automorphism of $R$. We also note that $R$ contains a Sylow $3$-subgroup of $G_{31}$, so the exception $G_{28}$ when $\ell=3$ in Theorem \ref{thm:sylowstablediagramautocomp} follows from the exception $G_{31}$ when $\ell=3$.
	\end{remark}
	
	\section*{Acknowledgements}
	The author would like to acknowledge Anthony Henderson for helpful discussions and suggestions. The author would would also like to express gratitude towards the anonymous referee for their useful comments and corrections. The research was supported by an Australian Government Research Training Program Scholarship.

	\bibliographystyle{alpha}
	\bibliography{references.bib}

\end{document}